\documentclass[11pt,a4paper]{amsart}

\usepackage{amsfonts, amsmath, amssymb, amsthm}
\usepackage[mathscr]{eucal} 
\usepackage{hyperref}

\usepackage{graphicx,color}
\usepackage{booktabs}
\usepackage{bbm}
\usepackage{xspace}
\usepackage[vlined,oldcommands]{algorithm2e}
\usepackage[latin1]{inputenc}
\usepackage[T1]{fontenc}

\theoremstyle{plain}
\newtheorem{theorem}{Theorem}
\newtheorem{proposition}[theorem]{Proposition}
\newtheorem{corollary}[theorem]{Corollary}
\newtheorem{lemma}[theorem]{Lemma}
\newtheorem*{question}{Question}

\theoremstyle{definition}

\newtheorem{remark}[theorem]{Remark}

\newcommand{\RR}{{\mathbb{R}}}

\newcommand{\cB}{{\mathcal{B}}}

\newcommand{\vones}{\mathbbm{1}}

\newcommand{\suchthat}{\, : \,}
\newcommand{\SetOf}[2]{\left\{#1\vphantom{#2}\suchthat\vphantom{#1}#2\right\}}
\newcommand\smallSetOf[2]{\{{#1}\suchthat{#2}\}}
\newcommand{\card}[1]{\lvert {#1} \rvert}
\newcommand{\order}[1]{O(#1)}

\DeclareMathOperator{\conv}{conv}

\DeclareMathOperator{\vertexset}{vert}
\DeclareMathOperator{\GL}{GL} 
\DeclareMathOperator{\PGL}{PGL} 

\renewcommand{\phi}{\varphi}

\newcommand{\bounded}[1]{\cB(#1)}       
\newcommand{\vertices}{\mathcal{V}}     
\newcommand{\facets}{\mathcal{F}}       
\newcommand{\hassediagram}{\mathcal{D}} 
\newcommand{\facelattice}[1]{\mathscr{L}({#1})}  
\newcommand{\projbnd}[1]{\overline{#1}} 
\newcommand{\farface}{F_{\infty}}       
\newcommand{\vertposet}[1]{\mathscr{P}({#1})}  
\newcommand{\sizefl}[1]{\phi({#1})}     
\newcommand{\sizeflP}{\phi}             
\newcommand{\sizeflprojbndP}{\overline{\phi}}
\newcommand{\sizeboundedP}{\phi'}
\newcommand{\sizeVertexPosetP}{\phi''}

\newcommand{\nverticesbndP}{\overline{n}}
\newcommand{\nfacetsbndP}{\overline{m}}
\newcommand{\nverticesP}{n}
\newcommand{\nfacetsP}{m}

\newcommand{\subsets}[1]{\mathscr{M}(#1)}

\newcommand{\polymake}{\texttt{polymake}\xspace}

\newenvironment{dense_itemize}{%
\begin{list}{$\circ$}%
{\setlength{\topsep}{1mm}%
\setlength{\partopsep}{0mm}%
\setlength{\parskip}{0mm}%
\setlength{\parsep}{0mm}%
\setlength{\itemsep}{0mm}%
\setlength{\labelwidth}{4mm}%
\setlength{\leftmargin}{0mm}%
\addtolength{\leftmargin}{\labelwidth}%
\addtolength{\leftmargin}{\labelsep}%
\setlength{\itemindent}{0mm}}}%
{\end{list}}

\begin{document}

\title[Computing Bounded Subcomplexes]{Computing the bounded subcomplex of an unbounded polyhedron}

\author[Herrmann, Joswig, and Pfetsch]{Sven Herrmann \and Michael Joswig \and Marc E. Pfetsch}
\address{Sven Herrmann, School of Computing Sciences, University of East Anglia, Norwich, NR4 7TJ, UK}
\email{s.herrmann@uea.ac.uk}
\address{Michael Joswig, Fachbereich Mathematik, TU Darmstadt, 64289 Darmstadt, Germany}
\email{joswig@mathematik.tu-darmstadt.de}
\address{Marc E. Pfetsch, Institute for Mathematical Optimization, TU
  Braunschweig, 38106 Braunschweig, Germany}
\email{m.pfetsch@tu-bs.de}

\begin{abstract}
  We study efficient combinatorial algorithms to produce the Hasse diagram
  of the poset of bounded faces of an unbounded polyhedron, given
  vertex-facet incidences. We also discuss the special case of simple
  polyhedra and present computational results.
\end{abstract}

\maketitle

\section{Introduction}

\noindent
The \emph{bounded subcomplex} of an (unbounded) convex polyhedron, is its
set of bounded (or, equivalently, compact) faces, partially ordered by
inclusion.  Polytopal complexes of this type arise in several situations
each of which is interesting in its own right.  Prominent examples are the
\emph{tight spans} of finite metric spaces of Dress~\cite{Dress},
see also Isbell~\cite{Is}, and the \emph{tropical polytopes} of Develin and
Sturmfels~\cite{DS}.  The purpose of this note is to present algorithms to
deal with such objects.

Typically, in applications a bounded complex of a polyhedron $P$ is given
only implicitly by a set of inequalities defining~$P$.  The primary goal is
to establish algorithms to make the bounded subcomplex explicit from this
input.  A direct approach is to enumerate the full face lattice
$\facelattice{\projbnd{P}}$ of a polytope~$\projbnd{P}$ projectively
equivalent to $P$ (which exists if $P$ is pointed) and to filter it to
obtain $\bounded{P}$, the poset of bounded faces.  However, since
$\facelattice{\projbnd{P}}$ often is much larger than $\bounded{P}$, this
is not efficient.  

One natural approach is to start with a (dual) convex hull computation which
also yields the vertex-facet incidences.  Here we focus on combinatorial
algorithms which take these vertex-facet incidences as input.  Other
approaches are discussed briefly.  Our Algorithm~\ref{algo:selective} in
Section~\ref{sec:SelectiveGeneration} is a modification of a combinatorial
algorithm for face lattice enumeration~\cite{KP}.  It uses the vertex-facet
incidences of~$\projbnd{P}$.  This method of \emph{selective generation} is
best possible, in the sense that its running time is linear in the size
of~$\bounded{P}$.  If only the vertex-facet incidences of~$P$ are given, that
is, no information about the unbounded edges is present, one can still
compute~$\bounded{P}$, although at the cost of a higher running time
(quadratic in the size of~$\bounded{P}$).  The corresponding
Algorithm~\ref{algo:moebius} is based on interleaving the computation of the
bounded faces with an incremental computation of the poset's M\"obius
function.

For the sake of simplicity of exposition we usually consider
complexity questions in the RAM model.  However, it is easy to modify
each of our results such that the bit complexity can be determined:
The only potential source of non-polynomiality in terms of RAM
complexity arises from calling LP type oracles.  Therefore, whenever
necessary, we explicitly mention the relevant sizes of the linear
programming problems which need to be solved.

If additional structural information is available, specialized algorithms come
into focus.  For instance, in the case where the polyhedron $P$ is simple, it
turns out that $\bounded{P}$ is determined by the (directed) vertex-edge-graph
of~$P$, see~\cite[Section~2]{HJ}.  This is an unbounded version of a result of
Blind and Mani~\cite{BM} obtained by applying techniques due to
Kalai~\cite{Ka}.  In Section~\ref{sec:Simple}, we employ the reverse-search
approach of Avis and Fukuda~\cite{AF} to generate the (directed) vertex-edge
graph of $\bounded{P}$ from an inequality description of~$P$, provided that
$P$ is simple.  This can be used to either construct the bounded faces or to
efficiently compute the face numbers of $P$ and~$\bounded{P}$.

In Section~\ref{sec:ComputationalResults}, computational experiments
with Algorithm~\ref{algo:selective} are presented. We investigate five
different cases in which it is interesting to compute the bounded
subcomplex. The computations show that some (surprisingly) large
instances can be handled.  The paper closes with a list of open
problems.

We are grateful to an anonymous referee for meticulous reading and for
requiring a correction in Algorithm~\ref{algo:moebius}.  This also contributed
to a cleaner description of this method.

\section{Preliminaries and Notation}
\label{sec:notation}

\noindent
Let $P$ be a polyhedron. The lattice of all faces (\emph{face lattice}) of $P$ is
denoted by~$\facelattice{P}$.  We define $\sizefl{P} =
\card{\facelattice{P}}$, that is, the number of all faces of~$P$. Moreover,
let~$\nverticesP$ be the number of vertices of~$P$ and~$\nfacetsP$ be its
number of facets.

A polyhedron is called \emph{pointed} if it does not contain an entire affine
line.  In the sequel we always assume that $P$ is pointed but unbounded.  For
our purposes this does not mean any loss of generality, since the bounded
subcomplex of a polyhedron is non-empty if and only if it is pointed.  In this
case, the polyhedron $P$ is projectively equivalent to a
polytope~$\projbnd{P}$.  Each such polytope is called a \emph{projective
  closure} of $P$.  A projective closure of $P$ has a special face~$\farface$
(the \emph{far face}) corresponding to the face at infinity.  Fixing an
admissible projective transformation $\gamma$ which maps $P$ to $\projbnd{P}$,
the face $\farface$ is the unique maximal face among the faces
of~$\projbnd{P}$ that are not images of faces of~$P$ under $\gamma$.  Note
that the combinatorial type of~$\farface$ depends on the geometry of~$P$, not
only its combinatorial structure, and that its dimension can be any number in
$\{0, \dots, \dim{P}-1\}$.

We begin with the description of an algorithm to compute the polytope~$\projbnd{P}$
from an inequality description of $P \subseteq \RR^d$. We refer to
Ziegler's monograph~\cite{Ziegler95} for a general discussion of admissible
projective transformations and~\cite[\S3.4]{JoswigTheobald08} for an
explicit construction. First, we compute one vertex $v$ of~$P$ by solving a
linear program similar to Phase I of the Simplex Algorithm.  Such a vertex
exists as $P$ is pointed.  In a second step we determine the active
constraints at $v$, that is, those inequalities which are satisfied with
equality at $v$.  Let $\tau$ be the affine transformation moving $v$ into
the origin.  In the image $\tau(P)$ the constraints active at $0$
correspond to homogeneous linear equations.  Any subset of those
constraints defines a polyhedral cone containing the translated polyhedron
$\tau(P)$.  Among the constraints active at~$0$ we choose a dual basis by
Gau\ss{}--Jordan-elimination.  Next we pick a linear transformation~$\rho$
mapping this dual basis to the basis which is dual to the standard basis of
$\RR^d$.  Then the image $\rho(\tau(P))$ is a polyhedron with vertex $0$
such that the boundary hyperplanes of the positive orthant define valid
inequalities.  In particular, $\rho(\tau(P))$ is contained in the positive
orthant.

In general, intersecting the image of a polyhedron in $\RR^d$ under a
projective transformation in $\PGL_d(\RR)$ with $\RR^d$ may yield
something non-convex.  However, if $M\in\GL_{d+1}(\RR)$ is an
invertible matrix with non-negative coefficients, then the induced
projective linear transformation $[M]\in\PGL_d(\RR)$ maps the positive
orthant into itself, and the image of any polyhedron inside is again a
polyhedron.  Now let $\mu$ be the unique projective linear
transformation which fixes each coordinate hyperplane of $\RR^d$ and
which additionally maps the hyperplane at infinity to the affine
hyperplane $\sum x_i=1$.  This transformation can be written as a
$(d+1)\times(d+1)$-matrix of zeroes and ones.  The polyhedron
\[
\projbnd{P} \ := \ \mu(\rho(\tau(P)))
\]
is contained in the simplex given by $\sum x_i\le 1$ and $x_i\ge 0$ for
$i\in[d]$.  In particular, $\projbnd{P}$ is bounded and projectively
equivalent to $P$.  The equation $\sum x_i=1$ defines a hyperplane
supporting $\projbnd{P}$, and its intersection with $\projbnd{P}$ is the
far face $\farface$.  For given matrix $A$ and right hand side $b$, let
$\ell(A,b)$ be the time complexity of solving the linear feasibility
problem to find a point~$x$ satisfying $Ax \le b$.  The discussion above can
be summarized as follows.

\begin{proposition}\label{prop:proj}
  If $P=P(A,b)$ is given by inequalities, there exists an explicit algorithm
  to compute~$\projbnd{P}$ in $O(\ell(A,b))$ time.
\end{proposition}

The unbounded edges of~$P$ correspond to edges of~$\projbnd{P}$ that contain
exactly one vertex in~$\farface$.  The vertices of~$P$ correspond to vertices
of~$\projbnd{P}$ that are not contained in~$\farface$.  Via these relations it
is easy to describe the face lattice of~$P$ in terms of~$\projbnd{P}$ and vice
versa.  The following relationship holds between the size $\phi=\sizefl{P}$ of
the face posets of $P$ and the size $\sizeflprojbndP = \sizefl{\projbnd{P}}$ of
its projective closure.

\begin{lemma}
  We have $\sizeflprojbndP\le 2(\sizeflP-1)$.
\end{lemma}

\begin{proof}
  For each unbounded face $F$ of $P$ let $s(F)$ be the intersection of $F$
  with the hyperplane at infinity.  Then $s$ is a map from the set of
  unbounded faces of $P$ to the set of non-empty faces of the far face
  $\farface$ which is surjective.  This proves that $\projbnd{P}$ has at
  most twice as many non-empty faces as~$P$.  To establish our slightly
  stronger claim, observe that $P$ must have at least one vertex, since we
  assumed $P$ to be pointed.
\end{proof}

The above bound is tight: For $\projbnd{P}$ consider any pyramid with its
basis as its special face.  In view of the obvious inequality
$\sizeflP\le\sizeflprojbndP$ the lemma shows that
$\sizeflprojbndP\in\Theta(\sizeflP)$; hence in statements about asymptotic
complexity, $\sizeflP$ and $\sizeflprojbndP$ can be used interchangeably.

We label the facets of~$\projbnd{P}$ from~$1$ to~$\nfacetsbndP$ and
identify each facet with its index in $\facets := \{1, \dots,
\nfacetsbndP\}$.  Similarly, label the vertices of~$\projbnd{P}$ from~$1$
to~$\nverticesbndP$, and identify each vertex with its index in $\vertices
:= \{1, \dots, \nverticesbndP\}$. Let $I \in \{0,1\}^{\nfacetsbndP
  \times \nverticesbndP}$ be a \emph{vertex-facet incidence matrix}
of~$\projbnd{P}$ with entries $I(f,v)$, that is, $I(f,v) = 1$ if facet~$f$ contains vertex~$v$,
and $I(f,v) = 0$ otherwise. Denote by $\alpha$ the number of incidences,
that is, the number of ones in~$I$. The face lattice of a polytope is
atomic and co-atomic: Every face of~$\projbnd{P}$ is uniquely defined by
its set of vertices and the set of facets it is contained in, respectively.
In the following, we will usually identify a face with its set of vertices
and store each such set by its characteristic vector (bitset). Hence, the
storage requirement per face is in~$\order{\nverticesbndP}$. Modifications
for other formats are straightforward -- see also
Section~\ref{sec:ComputationalResults} below.

The computation of an incidence matrix for a polytope given by inequalities
requires to generate the set of vertices, which can be exponentially many
in the number of inequalities.  There are several algorithms to generate
the vertices of a polytope, see Seidel~\cite{Sei97} for an overview and
Avis et al.~\cite{AviBS97} as well as \cite{Jos03} for discussions of the
complexity of vertex generation algorithms.  In fact, it is currently
unknown whether the vertices of a \emph{polytope} can be generated in
polynomial time in the combined size of the input and output.  It is,
however, known that it is NP-complete to decide whether the list of
vertices of a \emph{polyhedron} is complete (unbounded edges are ignored),
see Khachiyan et al.~\cite{KhaBBEG08} and Boros et al.~\cite{BorEGT10}. For
\emph{simple} $d$-polytopes, in which each vertex is contained in
exactly~$d$ facets, the \emph{reverse search} algorithm by Avis and
Fukuda~\cite{AF} produces the vertices in time $\order{d \cdot \nfacetsbndP
  \cdot \nverticesbndP}$.

Let~$\hassediagram$ be the \emph{Hasse diagram} of~$\projbnd{P}$. That is,
$\hassediagram$ is a directed rooted acyclic graph whose nodes correspond to
the elements of~$\facelattice{\projbnd{P}}$. If $N_H, N_G$ are nodes in
$\hassediagram$ and~$H$ and~$G$ are the corresponding faces of~$\projbnd{P}$,
then there is an arc $(N_H, N_G)$ in $\hassediagram$ if and only if $H \subset
G$ and $\dim(G) = \dim(H)+1$. Note that we distinguish between a face~$F$ and
its corresponding node~$N_F$ in $\hassediagram$, because this makes a
difference in Algorithms~\ref{algo:selective} and~\ref{algo:moebius} below.

\section{The Face Poset of the Bounded Subcomplex}\label{sec:poset}

\noindent
The partially ordered set (poset) of the bounded faces of an unbounded
polyhedron $P$ is denoted as $\bounded{P}$, and we call it the \emph{bounded
  subcomplex} of~$P$.  As a polytopal complex, $\bounded{P}$ is not pure in
general: it may have maximal faces of various dimensions.  The bounded
subcomplex is non-empty if and only if $P$ is pointed. We denote the size
of~$\bounded{P}$ by $\sizeboundedP$ in the following.

We primarily address the problem to compute~$\bounded{P}$, where~$P$ or
$\projbnd{P}$ is given in terms of a vertex-facet incidence matrix.  Other
input is discussed at the end of this section.  The output~$\bounded{P}$
should be given by its Hasse diagram, which is just a subgraph of the Hasse
diagram of~$\projbnd{P}$. Additionally or alternatively, the nodes, that is,
faces, might be labeled by either the set of facets the corresponding face
is contained in or by its set of vertices (this representation is unique,
since the faces are bounded); this requires an additional overhead of at
least~$\order{n}$ or~$\order{m}$ per face, respectively.

For the case where~$\sizeboundedP \approx \sizeflP$ and a vertex-facet
incidence matrix~$I \in \{0,1\}^{\nfacetsbndP \times \nverticesbndP}$
of~$\projbnd{P}$ is available, the direct approach is to apply the algorithm
of~\cite{KP} to generate the Hasse diagram of~$\projbnd{P}$
in~$\order{\nverticesbndP \cdot \alpha \cdot \sizeflP}$ time if the faces are
represented by their vertices; this algorithm will be modified in the next
section. We then remove the unbounded faces and their incident arcs by
checking whether the intersection of each face with~$\farface$ is empty or
not; if the faces are stored as bitsets, this requires
$\order{\nverticesbndP}$ time per face; the total running time for this step
is then dominated by the generation of the Hasse diagram of~$\projbnd{P}$.

Alternatively, one can also generate the Hasse diagram by working on the
dual, using $\order{\nfacetsbndP \cdot \alpha \cdot \sizeflP}$ time; in
this case, faces are represented by the facets they are contained in. In
this case, the removal step takes~$\order{\alpha}$ time per face, since we
have to intersect the vertex sets of~$\farface$ and of the facets that
contain the face (storing the vertex sets in sorted lists). If the
intersection is empty, the face is bounded. Thus, its running time is again
dominated by the generation of the Hasse diagram.

If $\sizeboundedP \ll \sizeflP$ there is a more efficient algorithm
available, which we describe in the next section.

\subsection{Selective Generation}
\label{sec:SelectiveGeneration}

In the following, we will describe an algorithm that computes (the Hasse
diagram of) $\bounded{P}$ efficiently if an incidence matrix of~$\projbnd{P}$
is available.  Our Algorithm~\ref{algo:selective} is a simple modification of the
algorithm in~\cite{KP}, which computes the face lattice of a polytope. We
provide some details for completeness and to enable a running time analysis.

\begin{algorithm}[tb]
  \dontprintsemicolon
  \linesnumbered
  
  \KwIn{incidence matrix $I$ of $\projbnd{P}$, far face~$\farface \subset \vertices$}
  \KwOut{Hasse diagram~$\hassediagram$ of $\bounded{P}$}

  initialize $\hassediagram$ with $N_{\varnothing}$
  corresponding to the empty face\;  \nllabel{line:init}
  %
  initialize the list $\mathcal{Q} \subseteq V(\hassediagram) \times
  2^\vertices$ by $(N_{\varnothing},\varnothing)$\;
  \While{$\mathcal{Q} \neq \varnothing$}{\nllabel{line:whileQ}
    choose some $(N_H,H) \in \mathcal{Q}$ and remove it from~$\mathcal{Q}$\;\nllabel{line:choose}
    compute the set~$\mathcal{G}$ of incident faces~$G$ with $\dim G = \dim H + 1$\;\nllabel{line:G}
    \ForEach{$G \in \mathcal{G}$}{\nllabel{line:ForG}
      \If{$G$ is bounded}{\nllabel{line:bounded}
        locate/create the node~$N_G$ corresponding to~$G$ in~$\hassediagram$\;\nllabel{line:locateG}
        \If{$N_G$ was newly created}{
          add~$(N_G,G)$ to~$\mathcal{Q}$\;
        }
        add the arc $(N_H,N_G)$ to~$\hassediagram$\;
      }
    }
  }
  \caption{The method of selective generation}
  \label{algo:selective}
\end{algorithm}

The algorithm performs a graph search through the Hasse diagram
of~$\bounded{P}$ starting from the bottom (the empty face). We denote by
$V(\hassediagram)$ the set of nodes of~$\hassediagram$. In each step we
consider a face~$H$ that has not been considered before, see
Step~\ref{line:choose}. Then we generate the set~$\mathcal{G}$ of all
faces~$G$ with $G \supset H$, $\dim G = \dim H + 1$, that is, there exists
an arc $(N_H,N_G)$ between the nodes $N_H$ and~$N_G$ corresponding to~$H$
and~$G$, respectively. It is shown in~\cite{KP} that the generation
of~$\mathcal{G}$ can be performed in $\order{\nverticesbndP^2} \subseteq
\order{\nverticesbndP \cdot \alpha}$ time. For each $G \in \mathcal{G}$, we
can now test whether~$G$ is unbounded by checking whether its intersection
with~$\farface$ is empty or not (in $\order{\nverticesbndP}$ time), see
Step~\ref{line:bounded}. If it is bounded, we proceed to generate the arc
$(N_H,N_G)$. For this, one has to find or create the node~$N_G$
corresponding to~$G$ (Step~\ref{line:locateG}). Here, we use a data
structure, called \emph{face tree} in~\cite{KP}, to store (bounded) faces
and their corresponding nodes.

Each edge of a face tree stores a vertex. The vertices along a path from
the root form a generating set for each face. Each sub-path corresponds to
a sub-face. Because each face of a bounded face is bounded as well, each
node of the tree corresponds to a different bounded face. The key point is
that the computation of the face from the generating set takes
$\order{\alpha}$ time, see~\cite{KP} for more details.  For future
reference we state the following.

\begin{lemma}[\cite{KP}]\label{lemma:face_tree}
  The face tree data structure allows to store bounded faces such that
  finding a face (or determining whether the face is not present) takes
  $\order{\alpha}$ time with a storage requirement of
  $\order{\nverticesbndP \cdot k}$, if~$k$ faces are stored.
\end{lemma}

Returning to Algorithm~\ref{algo:selective}, the while-loop in
Step~\ref{line:whileQ} is executed for each of the $\sizeboundedP$ bounded
faces, because of the boundedness test in Step~\ref{line:bounded}. Hence,
Steps~\ref{line:choose} and~\ref{line:G} contribute $\order{\nverticesbndP
  \cdot \alpha \cdot \sizeboundedP}$ time in total. The for-loop in
Step~\ref{line:ForG} is executed for each out-arc of a face
in~$\bounded{P}$ with respect to the Hasse diagram of~$\projbnd{P}$. Since
a face can have at most~$\nverticesbndP$ out-arcs, the for-loop is executed
at most $\nverticesbndP \cdot \sizeboundedP$ times.
Step~\ref{line:bounded} takes $\order{\nverticesbndP} \subseteq
\order{\alpha}$ time and Step~\ref{line:locateG} uses~$\order{\alpha}$
time, because of Lemma~\ref{lemma:face_tree}; note that since~$P$ is
pointed, each facet contains at least one vertex, i.e., $\nverticesbndP
\leq \alpha$. Thus, the for-loop contributes $\order{\nverticesbndP \cdot
  \alpha \cdot \sizeboundedP}$ time.  This shows the following.

\begin{theorem}\label{thm:selectiveGeneration}
  Given the vertex-facet incidences of~$\projbnd{P}$, the Hasse diagram
  of~$\bounded{P}$ can be computed in time $\order{\nverticesbndP \cdot
    \alpha \cdot \sizeboundedP}$.
\end{theorem}

\begin{remark}\label{rem:SelectiveGenerationSpace}
  By Lemma~\ref{lemma:face_tree}, the face tree uses $\order{\nverticesbndP
    \cdot \sizeboundedP}$ space. Storing faces as bitsets, the list
  $\mathcal{Q}$ needs at most~$\order{\nverticesbndP \cdot \sizeboundedP}$
  space. The output amounts to an additional space requirement
  of~$\order{\nverticesbndP \cdot \sizeboundedP}$.  Thus, we need a total
  amount of storage of $\order{\nverticesbndP \cdot \sizeboundedP}$.
\end{remark}

\begin{remark}
  Algorithm~\ref{algo:selective} is faster than the straight-forward
  algorithm described above if $\nverticesbndP \cdot \sizeboundedP <
  \min\{\nverticesbndP, \nfacetsbndP\} \cdot \sizeflP$.
\end{remark}

In Section~\ref{sec:Simple} below, we discuss the special case of simple
polyhedra.

\subsection{Selective Generation Without Knowing the Face at Infinity}
\label{sec:VertexPoset}

If only a vertex-facet incidence matrix~$I \in \{0,1\}^{\nfacetsP \times
  \nverticesP}$ of an unbounded (pointed) polyhedron~$P$ is known, but no
information about the unbounded edges is available, one can still produce (the
Hasse diagram of) $\bounded{P}$ as follows.

Let $\vertposet{P} := \{ \vertexset(F) \suchthat F $ proper face of $P\} \cup
\{\varnothing\}$ (where $\vertexset(F)$ is the set of vertices of $F$) be the
poset of the vertex sets of proper faces of~$P$.  It can be computed from any
vertex-facet incidence matrix of $P$, since it is the set of all non-empty
intersections of the subsets of $\vertices$ defined by the rows of the
incidence matrix (and additionally the empty set).  Note that this poset
contains the poset~$\bounded{P}$ as a subposet, but may contain additional
vertex sets of unbounded faces. It follows that $\sizeboundedP \le
\sizeVertexPosetP$, where $\sizeVertexPosetP$ is the size of~$\vertposet{P}$;
see Remark~\ref{rem:dwarfed_size} below for an example which shows that the
gap between~$\sizeboundedP$ and~$\sizeVertexPosetP$ can be large. Moreover,
because~$\bounded{P}$ is a polytopal complex, we have the following: if
$\vertexset{G} \subseteq \vertexset{F}$ for an unbounded face~$G$ and a
bounded face~$F$, there exists a bounded face~$H$ such that $\vertexset{H} =
\vertexset{G}$.

Using results of~\cite{JosKPZ01}, the boundedness of faces can be decided in the
following way. The \emph{M\"obius number} of a poset element~$S \in
\vertposet{P}$ is defined as
\[
\mu(S) = 
\begin{cases}
  1, & \text{if } S = \varnothing,\\
  - \sum\limits_{S' \subsetneq S} \mu(S'), & \text{otherwise},
\end{cases}
\]
where the sum ranges over all poset elements~$S' \in \vertposet{P}$
strictly less than~$S$. For a face~$H \neq P$, we define $\hat{\mu}(H) :=
\mu(\vertexset{H})$. For $H = P$, we add an artificial
top-element~$\hat{1}$ to $\vertposet{P}$ and define~$\hat{\mu}(P) :=
\mu(\hat{1})$. Then, the face~$H$ is bounded if and only if~$\hat{\mu}(H)
\neq 0$, see \cite[Corollary 4.5]{JosKPZ01}.  Once all poset elements are
known, the entire \emph{M\"obius function} of $\vertposet{P}$, that is, the
M\"obius numbers for all poset elements, can be computed by inverting an
appropriate $(\sizeVertexPosetP \times \sizeVertexPosetP)$-matrix (the
so-called $\zeta$-matrix) in $\order{(\sizeVertexPosetP)^3}$ time.

We obtain the following algorithm to compute the Hasse diagram
of~$\bounded{P}$. Since~$\vertposet{P}$ is atomic, its Hasse diagram can be
generated by the algorithm in~\cite{KP} in time $\order{\nverticesP \cdot
  \beta \cdot \sizeVertexPosetP}$, where $\beta$ is the number of vertex-facet
incidences of the polyhedron~$P$ (excluding unbounded information). Then the
vertex sets corresponding to unbounded faces are removed by using the M\"obius
function. In total, we obtain an $\order{\max\{\nverticesP \cdot \beta,
  (\sizeVertexPosetP)^2\} \cdot \sizeVertexPosetP}$ time algorithm. Since
$\vertposet{P}$ is also co-atomic, one can again apply the algorithm
of~\cite{KP} to the dual, which then yields an $\order{\max\{\nfacetsP \cdot
  \beta, (\sizeVertexPosetP)^2\} \cdot \sizeVertexPosetP}$ time algorithm.
One can often improve these running times as follows.

\begin{theorem}
  Given the vertex-facet incidences of~$P$ (without information on
  unbounded edges), the Hasse diagram of~$\bounded{P}$ can be computed in
  time $\order{\max\{\nverticesP^2 \cdot \sizeboundedP, \nverticesP \cdot
    \beta\} \cdot \nverticesP \cdot \sizeboundedP}$.
\end{theorem}

\begin{algorithm}[tb]
  \dontprintsemicolon
  \linesnumbered
  
  \KwIn{incidence matrix $I$ of $P$}
  \KwOut{Hasse diagram~$\hassediagram$ of $\bounded{P}$}
  
  initialize $\hassediagram$ with $N_{\varnothing}$
  corresponding to the empty face, $\subsets{\varnothing} = \{(\varnothing,1)\}$\;
  initialize queue~$\mathcal{Q} \subseteq V(\hassediagram) \times
  2^\vertices$ by $(N_{\varnothing},\varnothing)$\;
  \While{$\mathcal{Q} \neq \varnothing$}{\nllabel{line2:whileQ}
    remove first $(N_H,H)$ in $\mathcal{Q}$\;\nllabel{line2:choose}
    compute M\"obius number~$\hat{\mu}(H)$ using $\subsets{H}$\;\nllabel{line2:Moebius}
    \If{$\hat{\mu}(H) \neq 0$ ($H$ is bounded)}{\nllabel{line2:bounded}
      \ForEach{vertex $v$ not in $H$}{\nllabel{line2:ForG}
        $G$ $\leftarrow$ intersection of vertex sets of facets containing
        $H\cup\{v\}$ \; \nllabel{line2:closure}
        locate/create the node~$N_G$ corresponding to~$G$ in~$\hassediagram$\;\nllabel{line2:locateG}
        \If{$N_G$ was newly created}{
          add~$(N_G,G)$ to~$\mathcal{Q}$\nllabel{line2:createNew}\;
          $\subsets{G} \leftarrow \subsets{H}$\;\nllabel{line2:MoebiusUp1}
        }
        $\subsets{G} \leftarrow \subsets{G} \cup \{(H,\hat{\mu}(H))\}$\;\nllabel{line2:MoebiusUp2}
        add the arc $(N_H,N_G)$ to~$\hassediagram$\;
      }
    }
  }
  \caption{Method using the M\"obius function}
  \label{algo:moebius}
\end{algorithm}

\begin{proof}
  Algorithm~\ref{algo:moebius} is a modification of
  Algorithm~\ref{algo:selective}, in which the M\"obius numbers~$\hat{\mu}(H)$
  are computed on the fly in Step~\ref{line2:Moebius}. To this end, we use a
  set~$\subsets{H}$ that stores all elements of $\vertposet{P}$ strictly
  below~$H$ and their corresponding M\"obius numbers. Then
  Step~\ref{line2:Moebius} is a straight-forward summation.  Throughout this
  algorithm, the faces $G$ and $H$ of $P$ are encoded as sets of their
  vertices.

  To update~$\subsets{H}$ correctly, we perform a breadth-first search (BFS)
  over the Hasse diagram of~$\vertposet{P}$ from bottom to top, by
  organizing~$\mathcal{Q}$ in the algorithm as a queue.  We propagate the
  necessary information for the computation of M\"obius functions in
  Steps~\ref{line2:MoebiusUp1} and~\ref{line2:MoebiusUp2}.  It follows that
  once a face~$H$ leaves~$\mathcal{Q}$ in Step~\ref{line2:choose},
  $\subsets{H}$ contains all elements of $\vertposet{P}$ strictly below~$H$
  and their M\"obius functions.

  For the correctness of the entire algorithm it is essential that
  $\hat\mu(H)$ is computed correctly in Step~\ref{line2:Moebius}.  This is a
  consequence of two facts: First, $\hat\mu(H)$ is computed after all elements
  of $\vertposet{P}$ below $H$ have been processed; this is different from
  Algorithm~\ref{algo:selective}.  Second, the data structure $\subsets{H}$
  may also contain vertex sets of unbounded faces, but their contribution to
  $\hat\mu(H)$ is zero, by \cite[Corollary~4.5]{JosKPZ01}.  Therefore,
  $\hat\mu(H)$ evaluates to the M\"obius number of $H$ in the poset
  $\vertposet{P}$.

  If~$H$ is bounded, only its covering elements $G$ in $\vertposet{P}$ are
  generated in Step~\ref{line2:closure}.  More precisely, $G$ is generated
  from $H$ by adding one vertex $v$ not contained in $H$ and computing the
  \emph{closure} with respect to $I$; that is, $G$ is the intersection of all
  vertex sets of facets containing $H\cup\{v\}$; see~\cite[\S2.2]{KP}.  This
  is the only situation how new elements can enter the queue.  Hence, the
  while-loop in Step~\ref{line2:whileQ} is executed for each of the
  $\sizeboundedP$ bounded faces plus at most $\nverticesP \cdot \sizeboundedP$
  faces that are unbounded.

  The above discussion implies that $\card{\subsets{H}} \leq \nverticesP \cdot
  \sizeboundedP$, and hence Step~\ref{line2:Moebius} takes $\order{\nverticesP
    \cdot \sizeboundedP}$ time. Note that the size of~$\subsets{H}$
  is~$\order{\nverticesP^2 \cdot \sizeboundedP}$, as we store the vertex set
  of each face as a bitset.  The sets~$\subsets{G}$ can be organized using a
  face-tree. As observed in~{KP} Lemma~\ref{lemma:face_tree} can be used
  in this context as well, where we have to replace~$\alpha$ by~$\beta$
  and~$\nverticesbndP$ by~$\nverticesP$. To add information to~$\subsets{G}$
  in Step~\ref{line2:MoebiusUp2}, we need to check whether the subset~$H$ is
  present in~$\subsets{G}$, which can then be done in time~$\order{\nverticesP
    \cdot \beta}$. To copy~$\subsets{H}$ in Step~\ref{line2:MoebiusUp1}, we
  need~$\order{\nverticesP^2 \cdot \sizeboundedP}$ time.

  All other steps are as in Algorithm~\ref{algo:selective} and their analysis is
  similar as in the proof of Theorem~\ref{thm:selectiveGeneration}.  The only
  difference is that~$\alpha$ can be replaced by~$\beta$, that is, it is easy
  to see that by using the vertex-facet incidence matrix of~$P$, we only
  generate sets in~$\vertposet{P}$. In total, we obtain an
  $\order{\max\{\nverticesP^2 \cdot \sizeboundedP, \nverticesP \cdot \beta\}
    \cdot \nverticesP \cdot \sizeboundedP}$ time algorithm.
\end{proof}

\begin{remark}
  Each face-tree data structure needs $\order{\nverticesP^2 \cdot
    \sizeboundedP}$ space, see Remark~\ref{rem:SelectiveGenerationSpace}.  In
  total, the algorithm requires $\order{\nverticesP^3 \cdot
    (\sizeboundedP)^2}$ space for maintaining the sets~$\subsets{H}$. If faces
  are stored as bitsets, the queue $\mathcal{Q}$ needs at
  most~$\order{\nverticesP^2 \cdot \sizeboundedP}$ space. The output amounts
  to an additional space requirement of~$\order{\nverticesP \cdot
    \sizeboundedP}$.  Thus, we need a total amount of storage of
  $\order{\nverticesP^3 \cdot (\sizeboundedP)^2}$.
\end{remark}

\subsection{Polyhedra Given in Terms of Inequalities}

As mentioned above, the bounded subcomplex $\bounded{P}$ can be obtained by
removing from $\facelattice{\projbnd{P}}$ all \emph{unbounded} faces, that
is, faces that contain a vertex of~$\farface$.

If the defining inequalities of~$\projbnd{P}$ are given, that is,
$\projbnd{P} = \{x \suchthat A x \leq b\}$, the face lattice
of~$\projbnd{P}$ can be computed by an algorithm of Fukuda et
al.~\cite{FukLM97} in time $\order{\nfacetsbndP \cdot \ell(A,b) \cdot
  \sizeflP}$, where $\ell(A,b)$ is the time to solve a linear program with
input size equal to the size of~$A$ and~$b$\,; its space complexity is
$\order{\sizeflP \cdot \log \nfacetsbndP + p(A,b)}$, where $p(A,b)$ is the
space needed to solve a linear program of the size of~$A$ and~$b$.  This
algorithm outputs the faces as the sets of facets they are contained in. It
is also possible to apply the algorithm to the unbounded polyhedron~$P$.
Note that the algorithm does not produce the Hasse diagram, but it can be
computed using the algorithm in~\cite{KP}.

After the generation of $\facelattice{\projbnd{P}}$, one computes the set of
vertices of~$\farface$ and removes the unbounded faces. Additional work is
necessary if the faces of~$\bounded{P}$ should be given by their vertex sets.
This approach via the entire face lattice of~$\projbnd{P}$ is not efficient
when the bounded subcomplex $\bounded{P}$ is much smaller.  It is not obvious
to the authors whether or not the algorithm from~\cite{FukLM97} can be
modified to compute the bounded faces only.

The above algorithm avoids the explicit computation of the vertex-facet
incidences (with unknown complexity) and leads to a polynomial total time
algorithm if~$\sizeboundedP \approx \sizeflP$. In practice or if
$\sizeboundedP \ll \sizeflP$, Algorithm~\ref{algo:selective} might be
faster.

\subsection{Polyhedra Given in Terms of Vertices and Rays}

If the unbounded polyhedron is given by the list of its vertices and rays, one
can proceed analogously to the previous section by applying the algorithm of
Fukuda et al.~\cite{FukLM97} to the dual of~$\projbnd{P}$. It is also easy to
adapt their algorithm to work with vertices as input (their ``restricted face of
polyhedron'' problem can also be solved via a linear program in this case). The
faces are then given by their vertices and rays.  This approach has the same
drawbacks as the one in the previous section and can also be applied directly
to~$P$.

The affine hull of each bounded $k$-face of $P$ is spanned by $k+1$
affinely independent vertices.  For each $(k+1)$-tuple of the
$\nverticesP$ vertices of $P$, one can solve one linear program to
decide whether there exists a supporting hyperplane containing those
$k+1$ vertices and whose intersection with $P$ is $k$-dimensional.
This immediately yields the following.

\begin{proposition}
  Let $P$ be given in terms of vertices and rays, and let $\delta$ be a fixed
  constant.  Then the set of bounded faces up to dimension $\delta$ can be
  computed in $\order{\nverticesP^\delta\cdot\ell(\vertices,\text{\rm const})}$ time.
\end{proposition}

This algorithm does not produce the Hasse diagram of the $\delta$-skeleton
of~$\bounded{P}$ directly, but this can be achieved via the algorithm
in~\cite{KP}.  Alternative algorithms would result from computing the facets
and the corresponding incidences and then applying the previously mentioned
algorithms.

Note that Algorithms~\ref{algo:selective} and~\ref{algo:moebius} can easily
be modified to produce the $\delta$-skeleton as well.

\section{Simple Polyhedra}
\label{sec:Simple}

\noindent
In this section, we deal with the special case of simple polyhedra.  A
pointed $d$-dimensional polyhedron $P$ is \emph{simple} if each vertex is
contained in precisely~$d$ facets. Note that $\projbnd{P}$ may not be
simple, even if~$P$ is. Nevertheless, using a suitable generic construction
one can guarantee that the corresponding polytope is simple,
see~\cite[Prop.~2.2]{HJ}.

For simple polyhedra, Algorithm~\ref{algo:selective} can be implemented
more efficiently. Step~\ref{line:G} can be performed in time~$\order{d
  \cdot \alpha}$, and the for-loop is executed at most~$\order{d \cdot
  \sizeboundedP}$ times, see~\cite{KP} for more details.  Note that the
number of vertex-facet incidences is $\alpha = d \cdot \nverticesbndP$.
Thus, we obtain the following.

\begin{proposition}
  Given a simple polyhedron~$P$ and the vertex-facet incidences
  of $\projbnd{P}$, the (Hasse diagram of the) bounded subcomplex of~$P$
  can be computed in time $\order{d \cdot \alpha \cdot \sizeboundedP} =
  \order{d^2 \cdot \nverticesbndP \cdot \sizeboundedP}$ time.
\end{proposition}

If a simple polyhedron $P$ is given in terms of inequalities the reverse
search algorithm by Avis and Fukuda~\cite{AF} generates the vertices of~$P$
in~$\order{d \cdot \nfacetsP}$ time per vertex, where~$d$ is the dimension
of~$P$.  This is possible by performing a ratio test that decides whether
an edge is unbounded or not; see Avis~\cite{Avi00a} for details.  During
the vertex enumeration, the vertex-facet incidences of~$P$
and~$\projbnd{P}$ can be computed on-the-fly at no extra cost. Using the
incidences for~$\projbnd{P}$, we get the following.

\begin{corollary}
  Given a simple polyhedron~$P$ in inequality form, (the Hasse diagram of)
  $\bounded{P}$ can be computed in time $\order{\max \{\nfacetsbndP,\; d
    \cdot \sizeboundedP \} \cdot d \cdot \nverticesbndP}$.
\end{corollary}

Kalai gave an algorithm to compute the vertex-facet incidences of a simple
polytope from its vertex-edge graph~\cite{Ka}.  This method can be modified to
compute the \emph{$f$-vector} $(f_0,f_1,\dots,f_{d-1})$ of the bounded
subcomplex of a simple $d$-polyhedron, where $f_k$ is the $k$-th \emph{face
  number}, that is, the number of faces of dimension
$k$. Algorithm~\ref{algo:simple:f} presents this approach.

\begin{algorithm}[tb]
  \dontprintsemicolon
  \linesnumbered
  
  \KwIn{inequality description of a \emph{simple}, pointed, unbounded
    polyhedron $P$, far face $\farface$}
  \KwOut{$(f_0, \dots, f_d) = f$-vector of $\bounded{P}$}

  compute the vertices and the vertex-edge graph $\Gamma$ of $\projbnd{P}$\;\nllabel{line:simple:rs}
  find~$c$ such that linear program
  $\max\smallSetOf{c^{\rm T}x}{x\in P}$ takes its maximum on $\farface$ and which is generic on
  vertices of~$\projbnd{P}$\;\nllabel{line:simple:lp}
  direct the edges of $\Gamma$ along increasing $c$\;
  $d$ $\leftarrow$ $\dim P$\;
  \For{$k$ $\leftarrow$ $d,d-1,\dots, 0$}{
    $h_k$ $\leftarrow$ number of vertices of $P$ with out-degree $k$\; \nllabel{line:simple:h}
    $\overline{h}_k^\infty$ $\leftarrow$ number of vertices of $\farface$ with
    in-degree $k$\;
    $f_k$ $\leftarrow$ $\sum_{i=k}^d\tbinom{i}{k}(h_i-\overline{h}_i^\infty)$  \nllabel{line:simple:f}
  }
  \caption{Face numbers of simple polyhedra}
  \label{algo:simple:f}
\end{algorithm}

\begin{theorem}
  Algorithm~\ref{algo:simple:f} computes the face numbers of a simple
  polyhedron in $\order{d \cdot \nfacetsbndP \cdot \nverticesbndP}$ time.
\end{theorem}

\begin{proof}
  The reverse search algorithm~\cite{AF} produces the graph~$\Gamma$
  of~$\projbnd{P}$ in time $\order{d \cdot \nfacetsbndP \cdot
    \nverticesbndP}$. The running time of Step~\ref{line:simple:rs}
  dominates the remaining steps.

  Knowing the vertices of $\projbnd{P}$ and the far face is equivalent to
  knowing the vertices and rays of $P$.  Each ray describes a direction in
  which $P$ is unbounded.  Such a ray can be perturbed such that the
  corresponding linear objective function $c$ takes distinct values on
  distinct vertices (see Step~\ref{line:simple:lp}).

  The numbers $(h_0,h_1,\dots,h_d)$ computed in Step~\ref{line:simple:h}
  form the \emph{$h$-vector} of~$\projbnd{P}$, compare~\cite{Ka}.  Its
  relationship with the $f$-vector, expressed in Step~\ref{line:simple:f},
  is as follows: Each $k$-face of $\projbnd{P}$ has a unique minimal vertex
  with respect to $c$.  Conversely, each $k$-set of arcs which leave a
  fixed vertex $v$ spans a $k$-face such that $v$ the minimum with respect
  to $c$ is attained at $v$.  We have to ignore the unbounded $k$-faces of
  $P$, which are precisely the $k$-faces of $\projbnd{P}$ whose maximum
  with respect to $c$ is attained at some vertex of $\farface$.
\end{proof}

The formula in \cite[Prop.~2.4]{HJ} for the relationship between the $f$-
and $h$-vectors of an unbounded polyhedron is wrong.  The correct version
is in Step~\ref{line:simple:f} of Algorithm~\ref{algo:simple:f}.

\section{Computational Results}
\label{sec:ComputationalResults}

The following experiments were performed with the \polymake system \cite{GJ},
version 2.9.8.  The hardware used was an AMD Athlon 64 X2 Dual Core Processor
4200 (4435.84 bogomips) with 4GB main memory running Debian Linux.  We tested
Algorithm~\ref{algo:selective} requiring the vertex-facet incidences of
$\projbnd{P}$ as input only.  The timings for the convex hull computations
required are not given since discussing the various choices is a topic of its
own; see \cite{AviBS97,Jos03}.  Here we are focusing on the combinatorial
aspects.  Usually the convex hull computation takes much less time than the
computation of the face lattice.

The goals of our computations are threefold.
\begin{dense_itemize}
\item The limits of our algorithm can be estimated.
\item Several examples for which the computation of~$\bounded{P}$ is
  interesting are investigated and corresponding results are presented.
\item A rough estimation of the sizes up to which a computation
  of~$\bounded{P}$ is sensible, independent of the
  approach, can be derived (for our examples).
\end{dense_itemize}

Our theoretical analysis so far was based on the assumption that sets are
stored as bitsets; \polymake offers a suitable data type (wrapping an
implementation of the GMP~\cite{GMP}). However, the representation of sets
via balanced trees seems to be superior in a typical scenario.  This is the
one used in the tests below.

\subsection{Dwarfed Cubes}

The \emph{dwarfed $d$-cube} is the polytope
\[
\projbnd{D} := \SetOf{x\in[0,1]^d}{\sum x_i\le\tfrac{3}{2}} \, .
\]
The polytope~$\projbnd{D}$ has $\nfacetsbndP = 2d+1$ facets and
$\nverticesbndP = d^2+1$ vertices.  Moreover, $\alpha = d \cdot
\nverticesbndP$ as the polytope is simple.  The interest in these polytopes
comes from the fact that the dwarfed cubes provide difficult input to some
classes of convex hull algorithms, see Avis et al.~\cite{AviBS97}.  To
produce an unbounded polyhedron $D$, we send the \emph{dwarfing facet}
$\sum x_i=3/2$ to infinity by reversing the construction from
Proposition~\ref{prop:proj}.  The dwarfing facet contains $d(d-1)$
vertices; hence the bounded subcomplex has $\nverticesP=d^2+1-d(d-1)=d+1$
vertices and $\beta=\nverticesP\cdot d=d^2+d$ incidences between vertices
and facets.  The bounded subcomplex $\bounded{D}$ is a star-like graph with
$d+1$ nodes (and $d$ edges); that is, the number $\sizeboundedP$ of bounded
faces equals $2d+2$, including the empty face.

Table~\ref{tab:dwarfed_cubes} contains the results of the computations.  They
show that $\bounded{P}$ could be computed up to a very high dimension ($\geq
70$) with a moderate computing time; the key reason for this seems to be the
small number of bounded faces~$\sizeboundedP$.

\begin{remark}\label{rem:dwarfed_size}
  In order to count the faces of the (unbounded) dwarfed cube~$D$ we can
  employ Algorithm~\ref{algo:simple:f}: To this end consider the linear
  objective function $\sum x_i$ which is generic on the unbounded dwarfed
  cube, and which takes its maximum on the dwarfing facet.  This linear
  objective function gives a direction on each bounded edge or ray of $D$.
  The origin is the unique node of out-degree $d$, making up for $2^d$
  non-empty faces of $D$ whose minimum with respect to $\sum x_i$ is $0$.  The
  $d$ neighbors of $0$, the $d$ unit vectors, are nodes of out-degree $d-1$,
  making up for an additional $d\cdot 2^{d-1}$ non-empty faces.  Thus
  $\sizefl{D} = 2^d + d\cdot 2^{d-1} + 1$ (including~$\varnothing$ and~$D$).

  The poset $\vertposet{D}$ contains $2^d - 1$ elements for faces of the first
  kind above (not containing the face $D$ itself) and~$d$ elements
  corresponding to the~$d$ unit vectors.  Thus $\sizeVertexPosetP = 2^d + d$
  (including~$\varnothing$).  In contrast, $\sizeboundedP$ is only linear
  in~$d$.
\end{remark}

\begin{table}
  \caption{Results for Dwarfed cubes}
  \label{tab:dwarfed_cubes}
  \renewcommand{\arraystretch}{0.9}
  \begin{tabular*}{.8\linewidth}{@{\extracolsep{\fill}}rrrrrr@{}}
    \toprule
    \multicolumn{1}{r}{$d$} & \multicolumn{1}{r}{$\nfacetsbndP$} &
    \multicolumn{1}{r}{$\nverticesbndP$} & \multicolumn{1}{r}{$\alpha$} &
    \multicolumn{1}{r}{$\sizeboundedP$} & \multicolumn{1}{r@{}}{time (s)}\\
    \midrule
5 & 11 & 26 & 130 & 12 & 0.07\\
10 & 21 & 101 & 1010 & 22 & 0.11\\
15 & 31 & 226 & 3390 & 32 & 0.42\\
20 & 41 & 401 & 8020 & 42 & 1.87\\
25 & 51 & 626 & 15650 & 52 & 6.89\\
30 & 61 & 901 & 27030 & 62 & 28.35\\
35 & 71 & 1226 & 42910 & 72 & 72.93\\
40 & 81 & 1601 & 64040 & 82 & 159.64\\
45 & 91 & 2026 & 91170 & 92 & 324.86\\
50 & 101 & 2501 & 125050 & 102 & 593.73\\
55 & 111 & 3026 & 166430 & 112 & 1039.30\\
60 & 121 & 3601 & 216060 & 122 & 1743.58\\
65 & 131 & 4226 & 274690 & 132 & 2811.10\\
70 & 141 & 4901 & 343070 & 142 & 4457.96\\
75 & 151 & 5626 & 421950 & 152 & 6823.86\\
    \bottomrule
  \end{tabular*}
\end{table}

\subsection{Tight Spans of Metric Spaces}

The \emph{tight span} (or injective hull) $T_M$ of a finite metric space
$M : [d]\times [d]\to \RR$ (see Dress~\cite{Dress} and Isbell~\cite{Is}) is
defined as the bounded complex of the polyhedron
\[
P_M = \SetOf{x\in \RR^d}{x_i + x_j \geq M(i,j) \text{ for all } 1\leq i,j\leq
  d}\,.
\]
If~$M$ is generic enough, as in our examples, all inequalities define
facets, hence $\nfacetsbndP = \frac{d(d+1)}{2} + 1$ (including~$\farface$).
It was remarked by Sturmfels and Yu~\cite{SY} that $T_M$ is dual to the
complex of inner faces of the regular subdivision of the \emph{second
  hypersimplex}
\[
\Delta(2,d):=\conv\SetOf{e_i+e_j}{1\leq i\leq j\leq d},
\]
obtained from interpreting $M$ as a height function.  Tight spans are
relevant for applications in algorithmic biology, more precisely in
phylogenetics, cf.\ Dress et al.~\cite{DMT,Dress}.

Two ways to obtain special metric spaces and corresponding examples are
described next.

\subsubsection{Thrackle Metric}

A special triangulation of $\Delta(2,d)$, called the \emph{thrackle
  triangulation} was introduced by Stanley~\cite{Stanley} and thoroughly
investigated by De Loera et al.~\cite{DeLoeraSturmfelsThomas95}. It turns
out that the corresponding metric maximizes the number of faces of the
tight span \cite{HJ} and is equivalent to the tight span of the
\emph{maximal circular split system}, see Bandelt and
Dress~\cite[Section~3]{BandeltDress92}. From~\cite[Theorem~5.5]{HJ}, we
deduce that $\nverticesbndP=2^{d-1}+d$ and
\[
\sizeboundedP = \frac12\left((1-\sqrt 2)^d+(1+\sqrt 2)^d\right)+1.
\]
Note that the right hand side is always integral.

\begin{table}
  \caption{Results for thrackle metrics}
  \label{tab:ThrackleMetric}
  \renewcommand{\arraystretch}{0.9}
  \begin{tabular*}{.8\linewidth}{@{\extracolsep{\fill}}rrrrrr@{}}
    \toprule
    \multicolumn{1}{r}{$d$} & $\nfacetsbndP$ & $\nverticesbndP$ & $\alpha$ & $\sizeboundedP$ &time (s)\\
    \midrule
    3 & 7&7&24&8&0.01\\
    4 & 11&12 &60&18&0.01\\
    5 & 16&21&135&42&0.01\\
    6 & 22&38&288&100&0.03\\
    7 & 29&71&602&240&0.15\\
    8 & 37&136&1248&578&0.96\\
    9 & 46&265&2565&1394& 7.38\\
    10 & 56&522&5210&3364&61.90\\
    11 & 67&1035&10450&8120&559.08\\
    12 & 79&2060&20736&19602&5239.04\\
    13 & 92&4109&40820&47322&54302.46\\
    \bottomrule
  \end{tabular*}
\end{table}

Table~\ref{tab:ThrackleMetric} shows the results. Not surprisingly, it
turns out that with increasing dimension the the computation time
drastically increases along with the number of bounded faces.

\subsubsection{Random Metrics}

Our random metrics are obtained by taking the distances $M(i,j)$ to be
uniformly distributed in the interval $[1,2]$. Since this is generic (with
probability~$1$), we have $\nfacetsbndP=\frac{d(d+1)}{2} +1$. The sample
size for each dimension where 100 metrics. For $\nverticesbndP$, $\alpha$,
and $\sizeboundedP$, we state the arithmetic mean, and for the computation
time the mean together with the standard deviation.

\begin{table}
  \caption{Results for random metrics}
  \label{tab:RandomMetrics}
  \renewcommand{\arraystretch}{0.9}
  \begin{tabular*}{.9\linewidth}{@{\extracolsep{\fill}}rrrrrrr@{}}
    \toprule
    $d$ & $\nfacetsbndP$ & $\nverticesbndP$ & $\alpha$ & ${\sizeboundedP}$ & time (s) & stddev (s)\\\midrule
    5 & 16 & 21.00 & 135.00 & 42.00 & 0.03  & 0.00\\ 
    6 & 22 & 37.99 & 287.94 & 99.92 & 0.06  & 0.00\\ 
    7 & 29 & 70.65 & 599.55 & 237.20 & 0.19  & 0.01\\ 
    8 & 37 & 134.95 & 1247.60 & 568.96 & 1.17  & 0.09\\ 
    9 & 46 & 261.94 & 2609.46 & 1365.28 & 9.52  & 0.99\\ 
    10 & 56 & 513.55 & 5495.50 & 3275.68 & 86.69  & 9.69\\ 
    11 & 67 & 1008.46 & 11588.06 & 7802.88 & 841.02  & 109.44\\ 
    12 & 79 & 1997.28 & 24627.36 & 18709.52 & 9043.11  & 1351.32\\ 
    \bottomrule  
  \end{tabular*}
\end{table}

Table~\ref{tab:RandomMetrics} presents the results. Compared to the
thrackle metrics (Table~\ref{tab:ThrackleMetric}), random metrics have
 fewer bounded faces, but their computation times are larger.

\subsection{Tropical Polytopes}

Let $V=(v_{ik})$ be an $s \times t$-matrix with real coefficients.  We
define the polyhedron
\[
E_V \ := \ \SetOf{(u,w)}{u_i+w_k \le v_{ik}} \, ,
\]
where $u \in \RR^s$ and $w \in \RR^t$. Considering $(u,w)$ with
sufficiently small coordinates, one can see that $E_V$ is not empty.
Moreover, if $(u,w)\in E_V$ then for all $\lambda\in\RR$ we have
$(u+\lambda\vones,w-\lambda\vones)=(u,w)+\lambda(\vones,-\vones)\in E_V$.
Hence the one-dimensional subspace $\RR(\vones,-\vones)$ is contained in
the lineality space of $E_V$, and hence we can consider $E_V$ as a
polyhedron in the quotient $\RR^{s+t}/\RR(\vones,-\vones)$.  The polyhedron
$E_V$ in $\RR^{s+t}/\RR(\vones,-\vones)$ is pointed, and projecting its
bounded subcomplex to $\RR^t$ (or, alternatively, to $\RR^s$) yields the
\emph{tropical polytope} defined by $V$; see Develin and
Sturmfels~\cite{DS}.  The bounded subcomplex of $E_V$ is dual to the
regular subdivision of the product of simplices
$\Delta_{s-1}\times\Delta_{t-1}$, obtained by interpreting the matrix $V$
as a lifting function.  Here $\Delta_r$ denotes a simplex of dimension $r$.

We now translate the parameters for tropical polytopes into the parameters
that we used in our algorithms.  The dimension of the polyhedron $E_V$ (in
the quotient) equals $d=s+t-1$.  Its number $\nfacetsP$ of facets is less
than or equal to $s\cdot t$.  Throughout we have
$\nfacetsbndP=\nfacetsP+1$.  The number $\nverticesP$ of vertices satisfies
$\nverticesP\le\tbinom{s+t-2}{s-1}$.

\subsubsection{Tropical Cyclic Polytopes}

The \emph{tropical cyclic polytope} with parameters $(s,t)$ is given by the
matrix $V=(v_{ik})$ with $v_{ik} = i\cdot k$.  The corresponding
subdivision of $\Delta_{s-1}\times\Delta_{t-1}$ is known as the
\emph{staircase triangulation}; see Block and Yu \cite{BY}.  The polyhedron
$E_V$ is simple in this case.

\begin{table}
  \caption{Results for tropical cyclic polytopes}
  \label{tab:TropicalCyclic}
  \renewcommand{\arraystretch}{0.9}
  \begin{tabular*}{.8\linewidth}{@{\extracolsep{\fill}}rrrrrrr@{}}
    \toprule
    \multicolumn{1}{@{}r}{$(s,t)$} & \multicolumn{1}{r}{$d$} & \multicolumn{1}{r}{$\nfacetsbndP$} &
    \multicolumn{1}{r}{$\nverticesbndP$} & \multicolumn{1}{r}{$\alpha$} &
    \multicolumn{1}{r}{$\sizeboundedP$} & \multicolumn{1}{r@{}}{time (s)}\\
    \midrule
    (3,3) & 5 & 10 & 12 & 72 & 14 & 0.04\\
    (4,4) & 7 & 17 & 28 & 244 & 64 & 0.04\\
    (5,5) & 9 & 26 & 80 & 840 & 322 & 0.40\\
    (6,6) & 11 & 37 & 264 & 3144 & 1684 & 17.52\\
    (7,7) & 13 & 50 & 938 & 12614 & 8990 & 1198.35\\
    (8,8) & 15 & 65 & 3448 & 52392 & 48640 & 139091.23\\
    \addlinespace
    (3,10) & 12 & 31 & 68 & 1003 & 182 & 0.22\\
    (3,20) & 22 & 61 & 233 & 5903 & 762 & 9.17\\
    (3,30) & 32 & 91 & 498 & 17703 & 1742 & 110.16\\
    (3,40) & 42 & 121 & 863 & 39403 & 3122 & 814.14\\
    (3,50) & 52 & 151 & 1328 & 74003 & 4902 & 4418.48\\
    (3,60) & 62 & 181 & 1893 & 124503 & 7082 & 15595.14\\
    (3,65) & 67 & 196 & 2213 & 156653 & 8322 & 26858.18\\
    (3,70) & 72 & 211 & 2558 & 193903 & 9662 & 44651.71\\
    \bottomrule
  \end{tabular*}
\end{table}

\subsubsection{Tropical Permutohedra}

Each permutation $\sigma$ on the $t$ numbers from $0$ to $t-1$ can be
identified with the vector $(\sigma(0),\sigma(1),\dots,\sigma(t-1))$.  This
way each permutation contributes one row of a $(t!)\times t$-matrix $V$.  We
call the corresponding tropical polytope a \emph{tropical permutohedron}.  The
polyhedron~$E_V$ is not simple for $t\ge 3$.

\begin{table}[hbt]
 \caption{Tropical permutohedra}
  \label{tab:trop_permutohedra}
  \renewcommand{\arraystretch}{0.9}
  \begin{tabular*}{.8\linewidth}{@{\extracolsep{\fill}}rrrrrrr@{}}
    \toprule
    \multicolumn{1}{r}{$(s,t)$} & \multicolumn{1}{r}{$d$} & \multicolumn{1}{r}{$\nfacetsbndP$} &
    \multicolumn{1}{r}{$\nverticesbndP$} & \multicolumn{1}{r}{$\alpha$} &
    \multicolumn{1}{r}{$\sizeboundedP$} & \multicolumn{1}{r@{}}{time (s)}\\
    \midrule
    (6,3) & 8 & 19 & 24 & 261 & 50 & 0.05\\
    (24,4) & 17 & 97 & 152 & 6532 & 1424 & 9.07\\
    (120,5) & 124 & 601 & 1420 & 276725 & 76282 & 143535.58\\\bottomrule
  \end{tabular*}
\end{table}

Table~\ref{tab:TropicalCyclic} shows the results for tropical cyclic
polytopes and Table~\ref{tab:trop_permutohedra} for tropical permutohedra.
Both cases show a steep increase in computation time with increasing
dimension. The remarkable fact is that examples of these sizes can be
handled at all.

\section{Concluding Remarks and Open Questions}

\noindent
We presented combinatorial algorithms to compute (the Hasse diagram of) the
bounded faces of an unbounded pointed polyhedron.
Algorithm~\ref{algo:selective}, which takes the vertex-facet incidences
of~$\projbnd{P}$ as input, was also shown to work in practice via extensive
computations. It seems that the examples which we presented in the last
section show the limits of what one can currently compute in acceptable
time, unless some (possibly fundamentally) different idea comes up. For
instance, a central open question is the following.
\begin{question}
  Is there a polynomial total time algorithm to compute~$\bounded{P}$ for
  general $P$, given the inequalities?
\end{question}
For instance, this would follow if it were possible to modify the algorithm
of Fukuda et al.~\cite{FukLM97} to compute the bounded faces only.

\smallskip

In order to get an idea about the size of a bounded subcomplex it would be
interesting if it were possible to compute statistical information without
generating all faces.
\begin{question}
  Is there a polynomial time algorithm to compute the $f$-vector of
  $\facelattice{P}$ or $\bounded{P}$?
\end{question}
See the discussion in Section~\ref{sec:Simple} for simple polyhedra.

\bibliographystyle{amsplain}
\bibliography{bounded+subcomplex}

\providecommand{\bysame}{\leavevmode\hbox to3em{\hrulefill}\thinspace}
\renewcommand{\MR}{\relax\ifhmode\unskip\space\fi}
\providecommand{\MRhref}[2]{%
  \href{http://www.ams.org/mathscinet-getitem?mr=#1}{#2}
}
\providecommand{\href}[2]{#2}
\begin{thebibliography}{10}

\bibitem{Avi00a}
David Avis, \emph{A revised implementation of the reverse search vertex
  enumeration algorithm}, Polytopes -- Combinatorics and Computation (Gil Kalai
  and G{\"u}nter~M. Ziegler, eds.), DMV Seminar, vol.~29, Birkh{\"a}user, 
  Basel, 2000, pp.~177--198.

\bibitem{AviBS97}
David Avis, David Bremner, and Raimund Seidel, \emph{How good are convex hull
  algorithms?}, Comput. Geom. \textbf{7} (1997), no.~5-6, 265--301, 11th ACM
  Symposium on Computational Geometry (Vancouver, BC, 1995). \MR{MR1447243
  (98c:52017)}

\bibitem{AF}
David Avis and Komei Fukuda, \emph{A pivoting algorithm for convex hulls and
  vertex enumeration of arrangements and polyhedra}, Discrete Comput. Geom.
  \textbf{8} (1992), no.~3, 295--313, ACM Symposium on Computational Geometry
  (North Conway, NH, 1991). \MR{MR1174359 (93h:68137)}

\bibitem{BandeltDress92}
Hans-J{\"u}rgen Bandelt and Andreas W.~M. Dress, \emph{A canonical
  decomposition theory for metrics on a finite set}, Adv. Math. \textbf{92}
  (1992), no.~1, 47--105. \MR{MR1153934 (93h:54022)}

\bibitem{BM}
Roswitha Blind and Peter Mani-Levitska, \emph{Puzzles and polytope
  isomorphisms}, Aequationes Math. \textbf{34} (1987), no.~2-3, 287--297.
  \MR{MR921106 (89b:52008)}

\bibitem{BY}
Florian Block and Josephine Yu, \emph{Tropical convexity via cellular
  resolutions}, J. Algebraic Combin. \textbf{24} (2006), no.~1, 103--114.
  \MR{MR2245783}

\bibitem{BorEGT10}
Endre Boros, Khaled Elbassioni, Vladimir Gurvich, and Hans~Raj Tiwary,
  \emph{The negative cycles polyhedron and hardness of checking some polyhedral
  properties}, Annals of OR (2010), to appear.

\bibitem{DeLoeraSturmfelsThomas95}
Jes{\'u}s~A. De~Loera, Bernd Sturmfels, and Rekha~R. Thomas, \emph{Gr\"obner
  bases and triangulations of the second hypersimplex}, Combinatorica
  \textbf{15} (1995), no.~3, 409--424. \MR{MR1357285 (97b:13035)}

\bibitem{DS}
Mike Develin and Bernd Sturmfels, \emph{Tropical convexity}, Doc. Math.
  \textbf{9} (2004), 1--27 (electronic), erratum ibdm., pages 205--206.
  \MR{MR2054977 (2005i:52010)}

\bibitem{DMT}
Andreas Dress, Vincent Moulton, and Werner Terhalle, \emph{T-theory - an
  overview}, Europ. J. Combinatorics \textbf{17} (1995), 161--175.

\bibitem{Dress}
Andreas W.~M. Dress, \emph{Trees, tight extensions of metric spaces, and the
  cohomological dimension of certain groups: a note on combinatorial properties
  of metric spaces}, Adv. in Math. \textbf{53} (1984), no.~3, 321--402.
  \MR{MR753872 (86j:05053)}

\bibitem{FukLM97}
Komei Fukuda, Thomas~M. Liebling, and Fran{\c{c}}ois Margot, \emph{Analysis of
  backtrack algorithms for listing all vertices and all faces of a convex
  polyhedron}, Comput. Geom. \textbf{8} (1997), no.~1, 1--12. \MR{MR1452921
  (98b:68187)}

\bibitem{GJ}
Ewgenij Gawrilow and Michael Joswig, \emph{polymake: a framework for analyzing
  convex polytopes}, Polytopes---combinatorics and computation (Oberwolfach,
  1997), DMV Sem., vol.~29, Birkh\"auser, Basel, 2000, pp.~43--73.
  \MR{MR1785292 (2001f:52033)}

\bibitem{GMP}
Torbj\"orn Granlund et~al., \emph{The {GNU} multiple precision arithmetic
  library}, \url{gmplib.org}.

\bibitem{HJ}
Sven Herrmann and Michael Joswig, \emph{Bounds on the {$f$}-vectors of tight
  spans}, Contrib. Discrete Math. \textbf{2} (2007), no.~2, 161--184
  (electronic). \MR{MR2358269}

\bibitem{Is}
John~R. Isbell, \emph{Six theorems about injective metric spaces}, Comment.
  Math. Helv. \textbf{39} (1964), 65--76. \MR{MR0182949 (32 \#431)}

\bibitem{Jos03}
Michael Joswig, \emph{Beneath-and-beyond revisited}, Algebra, geometry, and
  software systems, Springer, Berlin, 2003, pp.~1--21. \MR{MR2011751
  (2004k:68169)}

\bibitem{JosKPZ01}
Michael Joswig, Volker Kaibel, Marc~E. Pfetsch, and G{\"u}nter~M. Ziegler,
  \emph{Vertex-facet incidences of unbounded polyhedra}, Adv. Geom. \textbf{1}
  (2001), no.~1, 23--36. \MR{MR1823950 (2002b:52015)}

\bibitem{JoswigTheobald08}
Michael Joswig and Thorsten Theobald, \emph{{Algorithmische Geometrie --
  Polyedrische und algebraische Methoden}}, Vieweg-Verlag, 2008.

\bibitem{KP}
Volker Kaibel and Marc~E. Pfetsch, \emph{Computing the face lattice of a
  polytope from its vertex-facet incidences}, Comput. Geom. \textbf{23} (2002),
  no.~3, 281--290. \MR{MR1927137 (2003h:52019)}

\bibitem{Ka}
Gil Kalai, \emph{A simple way to tell a simple polytope from its graph}, J.
  Combin. Theory Ser. A \textbf{49} (1988), no.~2, 381--383. \MR{MR964396
  (89m:52006)}

\bibitem{KhaBBEG08}
Leonid Khachiyan, Endre Boros, Konrad Borys, Khaled Elbassioni, and Vladimir
  Gurvich, \emph{Generating all vertices of a polyhedron is hard}, Discrete
  Comput Geom \textbf{39} (2008), 174--190.

\bibitem{Sei97}
Raimund Seidel, \emph{Convex hull computations}, Handbook of Discrete and
  Computational Geometry (Jacob Goodman and Joseph O'Rouke, eds.), CRC Press,
  Boca Raton, 1997, pp.~361--375.

\bibitem{Stanley}
Richard~P. Stanley, \emph{{E}ulerian partitions of the unit hypercube}, p.~49,
  D. Reidel, 1977.

\bibitem{SY}
Bernd Sturmfels and Josephine Yu, \emph{Classification of six-point metrics},
  Electron. J. Combin. \textbf{11} (2004), no.~1, Research Paper 44, 16 pp.
  (electronic). \MR{MR2097310 (2005m:51016)}

\bibitem{Ziegler95}
G{\"u}nter~M. Ziegler, \emph{Lectures on polytopes}, Springer-Verlag, New York,
  1995, Revised edition 1998.

\end{thebibliography}

\end{document}